\documentclass[11pt,reqno]{amsart}
\usepackage{}
\usepackage{esint}
\usepackage{bbm}
\usepackage{bm}
\usepackage{amssymb}
\usepackage{mathrsfs}
\usepackage{amsfonts}
\usepackage{amsfonts,amssymb,amsmath,amsthm}
\usepackage{url}
\usepackage{enumerate}
\usepackage[pdftex,bookmarksnumbered]{hyperref}
\usepackage{graphicx,xcolor}
\usepackage{pgfplots}
\usepackage{float}
\usepackage{caption}
\usepackage{ulem}
\usepackage{amsmath}
\usepackage{mathtools}

\newtheorem{theorem}{Theorem}[section]

\newtheorem{lemma}[theorem]{Lemma}
\newtheorem{lem}[theorem]{Lemma}

\newtheorem{prop}[theorem]{Proposition}
\newtheorem{corollary}[theorem]{Corollary}
\theoremstyle{definition}

\newtheorem{remark}[theorem]{Remark}

\numberwithin{equation}{section}

\newtheorem{ex}{Example}

\newenvironment{proofof}[1]{%
\renewcommand{\proofname}{{\rmfamily\bfseries Proof of Theorem #1}}% 证明标题加粗
\begin{proof}% 开始证明环境
}{%	
\end{proof}% 结束证明环境
}
\usepackage[left=2.5cm, right=2.5cm, top=2.0cm, bottom=2.0cm]{geometry}

\usepackage{graphicx}%
\usepackage{color}
\usepackage{cite}
\usepackage{hyperref}
\hypersetup{
	colorlinks = true,%
	citecolor = blue,%[rgb]{0.65,0.0,0.0},
	filecolor=red,%
	linkcolor = [rgb]{0.65,0.0,0.0},%
	anchorcolor = red,
	pagecolor = red,
	urlcolor= [rgb]{0.65,0.0,0.0}
	linktocpage=true,
	pdfpagelabels=true,
	bookmarksnumbered=true,
}

\makeatletter
\renewcommand{\abstractname}{Abstract} % 保持标题为"Abstract"，不自动大写
\renewenvironment{abstract}{%
	\begin{center}\bfseries \abstractname\end{center} % 摘要标题居中、加粗
	\small % 摘要正文字号
	\quotation % 摘要正文左右缩进
}{%
	\endquotation
}
\makeatother
\title{On Spherically Symmetric Sprays}
\author{ BENLING LI\textsuperscript{1}  and Yajing Gui\textsuperscript{1}}

\thanks{
Benling Li  is partly supported by the NNSFC(12471045),  NMNSF(2024J017) and K.C.
Wong Magna Fund in Ningbo University. }
\begin{document}
\maketitle
\maketitle
\begin{center}
    \textit{\textsuperscript{1}School of Mathematics and Statistics, Ningbo University, 315211 Ningbo, China}\\
    \textit{Emails: libenling@nbu.edu.cn (B. Li), gyj\_mathematics@163.com (Y. Gui)}
\end{center}

\vskip 10mm

\begin{abstract}
This paper studies spherically symmetric sprays, i.e., sprays that are invariant under orthogonal transformations.
We first establish a canonical form for such sprays, showing that their geodesic coefficients can be expressed as \(G^i = |y|\alpha(r,s) y^i + |y|^2\beta(r,s) x^i\), where \(r = |x|^2\) and \(s = \langle x,y\rangle/|y|\). For projectively flat spherically symmetric sprays -- which are directly related to Hilbert's fourth problem on characterizing metrics whose geodesics are straight lines -- we derive a complete classification of those with isotropic curvature, and in particular, we obtain the explicit form of those with zero curvature. Furthermore, we characterize sprays of weakly isotropic curvature in this class by a system of partial differential equations. These results may provide a unified framework for understanding symmetry and curvature in spray geometry and may offer new insights into the metrizability problem in Finsler geometry.
\end{abstract}

\vskip 5mm
\begin{quotation}

\noindent{\bf 2020 MR Subject Classification}: 53B40, 53C60.	
\end{quotation}
\section{Introduction}
Spray geometry investigates the properties of sprays on a manifold. Every Finsler metric induces a natural spray, but not all sprays arise from Finsler metrics. A \textit{spray} $G$ on an $n$-dimensional manifold $M$ is a smooth vector field on $TM\setminus\{0\}$. In  local coordinates \((x^i, y^i)\), it takes the form
   \begin{equation}\label{A1}
	G=y^i\frac{\partial}{\partial x^i}-2G^i\frac{\partial}{\partial y^i},
   \end{equation}
where  the spray coefficients \(G^i = G^i(x, y)\)  are smooth on \(TM \setminus \{0\}\) and locally homogeneous functions of degree two in $y$, and
are defined by a system of second order differential equations
   \begin{equation*}
	\frac{d^2x^i}{dt^2}+ 2 G^i(x,\frac{dx}{dt})=0.
   \end{equation*}

Spherical symmetry (rotational symmetry) has long played a distinguished role in differential geometry.
Classical examples include the Riemannian space forms of constant sectional curvature \(\mu\), with metric
\[
\alpha_\mu = \frac{\sqrt{|y|^2 + \mu\bigl(|x|^2|y|^2 - \langle x,y\rangle^2\bigr)}}{1+\mu|x|^2},
\]
whose induced spray satisfies \(G^i = -\frac{\mu\langle x,y\rangle}{1+\mu|x|^2} y^i\), where \(y\in T_x\Omega \simeq \mathbb{R}^n\) and $\Omega$ is an open subset in $\mathbb{R}^n$.
Hereafter, $|\cdot|$ and $\langle \cdot, \cdot\rangle$ denote the Euclidean norm and inner product, respectively.
In general relativity, the famous Schwarzschild metric also is spherically symmetric.
In Finsler geometry, spherically symmetric metrics have been systematically studied in recent decades.
A  Finsler metric $F$ on a domain $\Omega\subset \mathbb{R}^n$ is called {\it spherically symmetric} if
\[ F(Ux,Uy)= F(x,y), \qquad \forall\, U\in O(n), \]
where $O(n)$ is the $n$-dimensional orthogonal group.
Zhou \cite{11} first established the general form of such metrics and derived the expression for projectively spherically symmetric ones with constant flag curvature.
Huang and Mo \cite{3} obtained the equivalent equations for spherically symmetric Finsler metrics of scalar flag curvature, constructed many non-projectively flat examples, and revealed the connection between symmetry and curvature.
There are many interesting results on  spherically symmetric Finsler metrics. Some of them can be found in
\cite{MR4756608, MR4778999, MR4552590, MR5013825, 5, 6, 7,  MR4329062, MR4943539, Zhu2018}.
These results suggest that spherically symmetric sprays not only generalize classical models but also provide a unified framework for constructing Finsler metrics with prescribed curvature properties.

A \textit{spherically symmetric spray} is defined by the invariance
\begin{equation}\label{A3}
	G(Ux,Uy)= G(x,y),\qquad \forall\, U\in O(n).
\end{equation}
In this case the geodesic coefficients depend only on \(|x|,|y|,\langle x,y\rangle\) and are homogeneous of degree two in \(y\).
Our first main result gives a complete local description of such sprays.
By the spherical symmetry, the domain $\Omega\subset\mathbb{R}^n$ on which $G$ is defined can be taken to be either an open ball centered at the origin or the whole space $\mathbb{R}^n$.

\begin{theorem}\label{maintheorem1}
Let $G$ be a spray on a domain $\Omega \subset \mathbb{R}^n$.
Then $G$ is spherically symmetric if and only if there exist smooth functions $\alpha = \alpha(r, s)$ and $\beta = \beta(r,s)$, with
$r = |x|^2$ and $s = \langle x, y \rangle / |y|$, such that
\begin{align}\label{G^ialphabeta}
G^i(x,y)=|y|\alpha(r,s)y^i+|y|^2 \beta(r,s)x^i\quad(y\neq0),\qquad G^i(x,0)=0.
\end{align}
\end{theorem}
Thus spherically symmetric sprays are governed by two spherically symmetric functions.

A natural and important subclass consists of projectively flat sprays, which are intimately related to Hilbert's fourth problem.
A spray on \(\Omega\subset\mathbb{R}^n\) is \textit{projectively flat} if its geodesics are straight lines, i.e.
\[
G^i = P(x,y)y^i,
\]
where \(P(x,y)\) is homogeneous of degree one in \(y\) and  called {\it projective factor} of $G$.
Buc\u{a}taru and Muzsnay \cite{1} characterized when such a spray is metrizable by a Finsler metric.
Yang \cite{9} studied properties of projectively flat sprays and observed that many isotropic sprays are not Finsler metrizable.
In a related direction, Yang \cite{Yang2025} studied sprays on Hamel and Funk function models and obtained curvature and metrizability conditions for this class.
Li and Shen \cite{6} investigated sprays of isotropic curvature and showed that zero-curvature sprays can be induced by a family of Finsler metrics. 
For a projectively flat spherically symmetric spray, Theorem~\ref{maintheorem1} implies that the projective factor has the form
\begin{equation}\label{A4}
P(x,y)=|y|\,p(r,s).
\end{equation}

In 1926, L. Berwald extended the notion of Riemann curvature tensor to sprays \cite{2}.
The \textit{Riemann curvature tensor} $R^i_{\ k}$ is a linear transformation on the tangent space. In local coordinates, it can be expressed as
   \begin{equation}\label{B1}
	R^i_{\ k}=2  G^i_{x^k}- G^i_{x^j y^k} y^j+2 G^j G^i_{y^j y^k}- G^i_{y^j} G^j_{y^k}.
   \end{equation}
A spray $G$ is said to be of \textit{scalar curvature} if its Riemann curvature tensor $R^i_{\ k}$ can be expressed as
\begin{equation}\label{A5}
	R^i_{\ k}=R\delta^i_k-\tau_ky^i,
\end{equation}
where $R=R(x,y)$, $\tau_k=\tau_k(x,y)$ with $R=\tau_ky^k$. In particular, if
\begin{equation}\label{A6}
\tau_k=\frac{1}{2} R_{y^k},
\end{equation}
 then the spray $G$ is of \textit{isotropic curvature}. Furthermore, $G$ is called of {\it zero curvature} if $R^i_{\ k} =0$.

Our second main result classifies those sprays of isotropic curvature (and in particular zero curvature).
\begin{theorem}\label{maintheorem2}
Let \(G\) be a projectively flat spherically symmetric spray with \(G^i=|y|p(r,s)y^i\), where \(r=|x|^2\), \(s=\langle x,y\rangle/|y|\). Then \(G\) has isotropic curvature if and only if
\begin{equation}
p(r,s)=\left[\int\frac{u(r-s^2)}{s^2}\,ds+v(r)\right]s, \label{A7}
\end{equation}
where \(u\) and \(v\) are arbitrary smooth functions (with the integral chosen so that \(p\) is well-defined).
In particular, \(G\) has zero curvature if and only if either
\begin{enumerate}[(i)]
\item \(p\equiv0\) (the zero spray); or
\item \(p(r,s)=\dfrac{s\pm\sqrt{s^2-r+c}}{c-r}\), with \(c>0\) a constant,
\end{enumerate}
where the spray is defined on $B_{\sqrt{c}}^n(0)\times(\mathbb{R}^n\backslash\{0\})$. In the latter case the expression is equivalent to the previous one with \(u(r-s^2)=\mp\frac{1}{\sqrt{s^2-r+c}}\) and an appropriate \(v(r)\) (see Remark \ref{remark}).
\end{theorem}

Formula \eqref{A7} allows one to construct many sprays of isotropic curvature by suitable choices of \(u\) and \(v\).
Examples of both zero and non-zero isotropic curvature are given in Corollary~\ref{example1} and Example~\ref{example2} in Section~\ref{section4}.

 In \cite{Cheng2022}, Cheng, Cao and Qing studied the relationship between \(\chi\)-curvatures of projectively related sprays and used it to construct sprays of isotropic curvature, including examples that are not Finsler metrizable.
In \cite{Yang2023}, Yang obtained the necessary and sufficient conditions for sprays of isotropic curvature to be Finsler metrizable. Beyond isotropic curvature, a more general notion -- \textit{weakly isotropic curvature} -- has recently been introduced \cite{8}.
For spherically symmetric sprays, the conditions simplify considerably, and we derive an equivalent system of PDEs for the function \(p(r,s)\).
However, the detailed analysis of weakly isotropic curvature is technically more involved; we postpone it to Section~\ref{section5}, where we also present some explicit special solutions.

We note that a companion paper by Jia and Li \cite{JiaLi} (Li being a coauthor of the present paper), which cites the present work and was completed after it, builds on the canonical form of Theorem~\ref{maintheorem1} to obtain a classification of spherically symmetric sprays of isotropic curvature. The main result of the present paper is precisely this canonical form, Theorem~\ref{maintheorem1}; on its basis we further restrict to the projectively flat subclass and obtain the complete classification (Theorem~\ref{maintheorem2}), the explicit zero-curvature sprays (Proposition~\ref{prop4.2}), and their Finsler metrizability (Corollary~\ref{example1}), which are not contained in \cite{JiaLi}.

The paper is organized as follows. Section~\ref{section3} proves Theorem~\ref{maintheorem1} (the canonical form of spherically symmetric sprays).
Section~\ref{section4} proves Theorem~\ref{maintheorem2} (classification of isotropic and zero curvature), and also contains the necessary curvature computations.
Section~\ref{section5} is devoted to weakly isotropic curvature: we derive the governing PDEs, prove Theorem~\ref{maintheorem3}, and give examples.

Throughout the paper, \(G\) denotes a spray on a domain $\Omega \subset \mathbb{R}^n$, while \(K\) (and sometimes \(H\)) denotes compact Lie groups.
Subscripts indicate partial derivatives, e.g. \(p_r = \partial p/\partial r\), \(p_{rs} = \partial^2 p/\partial r\partial s\), etc.

\section{Spherically symmetric spray}\label{section3}
In this section we establish the canonical form of spherically symmetric sprays.
Using Schwarz's theorem, the geodesic coefficients depend only on \(|x|^2,|y|^2,\langle x,y\rangle\).
Together with homogeneity, we derive \(G^i = |y|\alpha(r,s)y^i + |y|^2 \beta(r,s)x^i\) (\(r=|x|^2\), \(s=\langle x,y\rangle/|y|\)), proving Theorem~\ref{maintheorem1}.
As a consequence, we also obtain a criterion for such a spray to be induced by a Finsler metric when it has non-zero isotropic curvature.

The following theorem was proved by G. W. Schwarz in 1975.
\begin{theorem}\cite[Theorem 1]{Schwarz} \label{Schwarz}
Let \(K\) be a compact Lie group acting orthogonally on  a finite-dimensional real vector space \(V\).
Let \(\rho_1, \dots, \rho_\ell\)\footnote{The number $\ell$ of generators is not directly related to
$\dim V = n$; it can be greater than, equal to, or less than
$n$ depending on the group action.}
 be generators of the algebra \(P[V]^K\) of \(K\)-invariant polynomials, and let $\rho=(\rho_1, \dots, \rho_\ell) : V \rightarrow \mathbb{R}^\ell$.
Then the map
\[
\rho^*: C^\infty(\mathbb{R}^\ell) \longrightarrow C^\infty(V)^K,\quad
f \longmapsto f(\rho_1, \dots, \rho_\ell)
\]
is surjective. Consequently, every smooth \(K\)-invariant function on \(V\) can be written as a smooth function
of these basic polynomial invariants.
\end{theorem}

\begin{remark}[Application to our setting]
In the special case of the diagonal action of \(O(n)\) on \(\mathbb{R}^n \times \mathbb{R}^n\),
the algebra of invariant polynomials is generated by
\[
\rho_1 = |x|^2,\quad \rho_2 = |y|^2,\quad \rho_3 = \langle x, y \rangle.
\]
Schwarz's theorem therefore guarantees that any smooth \(O(n)\)-invariant function
\(f(x,y)\) can be expressed as
\[
f(x,y) = \tilde{f}\bigl(|x|^2, |y|^2, \langle x, y \rangle\bigr)
\]
for some smooth function \(\tilde{f}\) on \(\mathbb{R}^3\).
For a comprehensive treatment of invariant smooth functions,
see Michor~\cite[Chapter 4]{Michor}.

By this fact, the following corollary is a direct consequence of Theorem \ref{Schwarz}.
Alternately, a direct proof also can be obtained by a similar method in \cite{11}, where L. Zhou studied spherically symmetric Finsler metrics.
\end{remark}

\begin{corollary}\label{Cor1}
Let \(f: \mathbb{R}^n \times \mathbb{R}^n \to \mathbb{R}\) be a smooth function that is invariant under the diagonal action of the orthogonal group \(O(n)\):
\[
f(Ux, Uy) = f(x,y) \qquad \forall\, U \in O(n), \; \forall\, (x,y) \in \mathbb{R}^n \times \mathbb{R}^n.
\]
Then there exists a smooth function \(\tilde{f}: \mathbb{R}^3 \to \mathbb{R}\) such that
\[
f(x,y) = \tilde{f}\bigl( |x|^2, |y|^2, \langle x, y \rangle \bigr) \quad \forall (x,y).
\]
\end{corollary}

In the following, we prove two lemmas which compose the necessity and sufficiency of Theorem~\ref{maintheorem1}.
\begin{lemma}\label{lem_spherical-spray_necessi}
Let $G$ be a spray on a domain $\Omega \subset \mathbb{R}^n$.
If $G$ is spherically symmetric, then there exist two functions $\alpha = \alpha(r, s)$ and $\beta = \beta(r,s)$, where
$r = |x|^2$ and $s = \langle x, y \rangle / |y|$, such that
\begin{align}\label{G^ialphabetanecessity}
G^i = |y| \alpha(r,s) y^i + |y|^2 \beta(r,s) x^i, \text{ for } y\neq 0.
\end{align}
\end{lemma}
\begin{proof}

Identify the tangent bundle $T\Omega$ with \(\Omega \times\mathbb{R}^n\); local coordinates are \((x^i,y^i)\).
Then the  spray \(G\) is a smooth vector field on \(T\Omega\) given by \eqref{A1} as
\[
G = y^i\frac{\partial}{\partial x^i} - 2G^i(x,y)\frac{\partial}{\partial y^i},
\]
where the coefficients \(G^i(x,y)\) are smooth and positively homogeneous of degree $2$ in \(y\):
\(G^i(x,\lambda y)=\lambda^2 G^i(x,y)\) for all \(\lambda>0\).

For an orthogonal transformation \(U = (U^i_j) \in O(n)\) define the diffeomorphism
\[
\Psi_U : T\Omega\longrightarrow T\Omega,\qquad \Psi_U(x,y)=(Ux,Uy), \quad \forall\, (x,y)\in T\Omega.
\]
Since $G$ is spherically symmetric,
 \(\Psi_U\) preserve the spray:
\[
(\Psi_U)_* G = G, \qquad  \forall\, U\in O(n).
\]
This means that for any point \(p=(x,y)\in TM\),
\[
(\Psi_U)_*\bigl(G|_{\Psi_U^{-1}(p)}\bigr)=G|_p .
\]
Set \(q:=\Psi_U^{-1}(p)=(U^{-1}x,U^{-1}y)\).
Then the spray  at \(q\) is
\[
G_q = (U^{-1}y)^j\frac{\partial}{\partial x^j}\Big|_{q}
      -2G^j(q)\frac{\partial}{\partial y^j}\Big|_{q},
\]
where \(G^j(q)=G^j(U^{-1}x,U^{-1}y)\).

In coordinates \(\Psi_U\) reads
\[
(\Psi_U)^i(x,y)=(Ux)^i = U^i_j x^j,\qquad
(\Psi_U)^{n+i}(x,y)=(Uy)^i = U^i_j y^j.
\]
Hence the Jacobian matrix is constant:
\[
\frac{\partial (\Psi_U)^i}{\partial x^j}=U^i_j,\quad
\frac{\partial (\Psi_U)^i}{\partial y^j}=0,\quad
\frac{\partial (\Psi_U)^{n+i}}{\partial x^j}=0,\quad
\frac{\partial (\Psi_U)^{n+i}}{\partial y^j}=U^i_j.
\]
Consequently, for any tangent vector \((\xi,\eta)\) (horizontal component \(\xi\), vertical component \(\eta\)),
\[
d\Psi_U|_{(x,y)}(\xi,\eta)=(U\xi,U\eta).
\]
In particular, for the basis vectors:
\[
d\Psi_U|_q\Bigl(\frac{\partial}{\partial x^j}\Big|_{q}\Bigr)=U^i_j\frac{\partial}{\partial x^i}\Big|_{\Psi_U(q)}=U^i_j\frac{\partial}{\partial x^i}\Big|_{p},
\qquad
d\Psi_U|_q\Bigl(\frac{\partial}{\partial y^j}\Big|_{q}\Bigr)=U^i_j\frac{\partial}{\partial y^i}\Big|_{p}.
\]
Now
\begin{align*}
(\Psi_U)_*G|_p
 & = d\Psi_U|_q(G_q)
   = (U^{-1}y)^j\, d\Psi_U|_q\Bigl(\frac{\partial}{\partial x^j}\Big|_{q}\Bigr)
   -2G^j(q)\, d\Psi_U|_q\Bigl(\frac{\partial}{\partial y^j}\Big|_{q}\Bigr) \\
 & = (U^{-1}y)^j U^i_j\frac{\partial}{\partial x^i}\Big|_{p} -2G^j(q)\,U^i_j\frac{\partial}{\partial y^i}\Big|_{p}.
\end{align*}
In the first term, \((U^{-1}y)^jU^i_j = (UU^{-1}y)^i = y^i\), so it becomes \(y^i\frac{\partial}{\partial x^i}\big|_{p}\).
In the second term, \(-2G^j(q)U^i_j = -2U^i_j G^j(U^{-1}x,U^{-1}y)\).  Thus
\[
(\Psi_U)_*G|_p = y^i \frac{\partial}{\partial x^i}\Big|_{p}
                -2 U^i_j G^j(U^{-1}x,U^{-1}y) \frac{\partial}{\partial y^i}\Big|_{p}.
\]

On the other hand,
the spray at \(p\) itself is
\[
G|_p = y^i\frac{\partial}{\partial x^i}\Big|_{p} -2G^i(x,y)\frac{\partial}{\partial y^i}\Big|_{p}.
\]
Invariance \((\Psi_U)_*G|_p = G|_p\) must hold for every \(p\).  The horizontal parts are already equal, so we compare the coefficients of \(\frac{\partial}{\partial y^i}\):
\[
U^i_jG^j(U^{-1}x,U^{-1}y)=G^i(x,y).
\]
Replacing \(U\) by \(U^{-1}\) (which is also orthogonal) gives the equivalent form
\begin{align}\label{G^i(Ux,Uy)}
G^i(Ux,Uy)=U^i_j\,G^j(x,y).
\end{align}

Now we prove
the vector $\mathbb{G}(x,y) = (G^1(x,y),\dots,G^n(x,y))$ must belong to the linear span of \(x\) and \(y\).
Denote
\(V = \operatorname{span}\{x, y\} \subset \mathbb{R}^n.\)
We shall prove that \(\mathbb{G}(x,y) \in V\).

Take any non-zero vector \(z \in V^\perp\), i.e.,
\[
\langle z, x \rangle = 0, \quad \langle z, y \rangle = 0.
\]
Our aim is to show \(\langle \mathbb{G}(x,y), z \rangle = 0\), which implies \(G(x,y) \perp V^\perp\) and hence \(\mathbb{G}(x,y) \in (V^\perp)^\perp = V\).

Because \(z\) is orthogonal to both \(x\) and \(y\), one can define an orthogonal map \(\tilde{U} \in O(n)\) as follows:
\begin{itemize}
  \item On the subspace \(V\), let \(\tilde{U}\) be the identity: \(\tilde{U}|_V = \mathrm{Id}\).
  \item On the line \(\mathbb{R}z\) (which lies in \(V^\perp\)), set \(\tilde{U} = -\mathrm{Id}\); i.e. \(\tilde{U} z = -z\).
  \item On the orthogonal complement of \(\mathbb{R}z\) inside \(V^\perp\) (i.e. on \(\{ w \in V^\perp : w \perp z \}\)), let \(\tilde{U}\) be the identity.
\end{itemize}
Then
\[
\tilde{U}x = x,\qquad \tilde{U}y = y,\qquad \tilde{U}z = -z.
\]
Combining with \eqref{G^i(Ux,Uy)}, it follows that
\begin{align}\label{G^i_verctor}
\langle \mathbb{G}(x,y), z \rangle =\langle \mathbb{G}(\tilde{U}x,\tilde{U}y), z \rangle = \langle \tilde{U} \mathbb{G}(x,y), z \rangle.
\end{align}
Using the fact that orthogonal transformations preserve the inner product, we have
\[
\langle \tilde{U} \mathbb{G}(x,y), z \rangle = \langle \mathbb{G}(x,y), \tilde{U}^T z \rangle.
\]
Because \(\tilde{U}\) is orthogonal, \(\tilde{U}^T = \tilde{U}^{-1}\).  Since \(\tilde{U}^2 = I\) on the whole space (it is a reflection across the hyperplane \(V^\perp \cap \{z\}^\perp\)), we have \(\tilde{U}^{-1} = \tilde{U}\), so \(\tilde{U}^T z = \tilde{U} z = -z\). Consequently,
\[
\langle \tilde{U} \mathbb{G}(x,y), z \rangle = \langle \mathbb{G}(x,y), -z \rangle = -\langle \mathbb{G}(x,y), z \rangle.
\]
Equation \eqref{G^i_verctor} therefore yields
\[
\langle \mathbb{G}(x,y), z \rangle = -\langle \mathbb{G}(x,y), z \rangle,
\]
which implies
\[
\langle \mathbb{G}(x,y), z \rangle = 0.
\]
Therefore \(\mathbb{G}(x,y) \in (V^\perp)^\perp = V\).

For \eqref{G^ialphabetanecessity}, since \(\mathbb{G}(x,y) \in  V = \operatorname{span}\{x,y\}\), there exist two smooth functions $a(x,y)$ and $b(x,y)$ such that
\[ G^i(x,y) = a(x,y) y^i + b(x,y) x^i.\]
Together with \eqref{G^i(Ux,Uy)}, it follows that
\[
G^i(Ux,Uy) = a(Ux,Uy) U^i_j y^j + b(Ux,Uy) U^i_j x^j = U^i_j ( a(x,y) y^j + b(x,y) x^j ) =  a(x,y) U^i_j y^j + b(x,y)  U^i_j x^j.
\]
Then
\[ a(Ux,Uy) = a(x,y), \qquad b(Ux,Uy) = b(x,y). \]
Since $U\in O(n)$, by Corollary \ref{Cor1} there exist smooth functions $\tilde{a}(|x|^2,|y|^2,\langle x,y\rangle)$ and $\tilde{b}(|x|^2,|y|^2,\langle x,y\rangle)$ such that
\begin{align}\label{G^iab}
G^i(x,y)=\tilde{a}(|x|^2,|y|^2,\langle x,y\rangle) \,y^i + \tilde{b}(|x|^2,|y|^2,\langle x,y\rangle) \,x^i.
\end{align}

Since $G^i(x,y)$ is positively homogeneous of degree 2 in  \(y\):
\[
G^i(x,\lambda y)=\lambda^2 G^i(x,y), \qquad \forall\, \lambda > 0.
\]
Applying this to \eqref{G^iab} gives
\[
\tilde{a}(|x|^2,\lambda^2|y|^2,\lambda\langle x,y\rangle)\, \lambda y^i
   +\tilde{b}(|x|^2,\lambda^2|y|^2,\lambda\langle x,y\rangle)\,x^i
   =\lambda^2 \tilde{a}(|x|^2,|y|^2,\langle x,y\rangle) y^i + \lambda^2 \tilde{b}(|x|^2,|y|^2,\langle x,y\rangle) x^i.
\]
Since \(x^i\) and \(y^i\) are linearly independent for generic points, we have
\[
\tilde{a}(|x|^2,\lambda^2|y|^2,\lambda\langle x,y\rangle) = \lambda \tilde{a}(|x|^2,|y|^2,\langle x,y\rangle), \qquad
\tilde{b}(|x|^2,\lambda^2|y|^2,\lambda\langle x,y\rangle) = \lambda^2  \tilde{b}(|x|^2,|y|^2,\langle x,y\rangle).
\]

Then for $y\neq 0$, by setting $\lambda=1/|y|$ (which is well-defined since $y\neq0$), $r=|x|^2$ and $s=\langle x,y\rangle/|y|$, the functions $\tilde{a}$ and $\tilde{b}$ can be written as
\[
\tilde{a}(|x|^2,|y|^2,\langle x,y\rangle) = |y| \alpha(r,s),
\qquad
\tilde{b}(|x|^2,|y|^2,\langle x,y\rangle) = |y|^2 \beta(r,s),
\]
where
\[
\alpha(r,s):=\tilde{a}(r,1,s),
\qquad
\beta(r,s):=\tilde{b}(r,1,s).
\]
Since $\tilde{a}$ and $\tilde{b}$ are smooth on $\mathbb{R}^3$, the functions $\alpha$ and $\beta$ are smooth on $\mathbb{R}^2$, in particular at $s=0$.
Moreover, as $y\to 0$,
\[
|y|\alpha(r,s)y^i+|y|^2\beta(r,s)x^i=O(|y|^2)\to 0,
\]
and the extension by $G^i(x,0)=0$ is continuous, with $G^i$ smooth on $TM\setminus\{0\}$, since $\alpha,\beta$ are smooth and $|y|$ is smooth away from $y=0$.
This together with \eqref{G^iab} yields \eqref{G^ialphabetanecessity}.
\end{proof}

\begin{lemma}\label{lem_spherical-spray_suffi}
Let \(\alpha, \beta: \mathbb{R}^2 \to \mathbb{R}\) be smooth functions. For \((x,y) \in \mathbb{R}^n \times (\mathbb{R}^n \setminus \{0\})\), define
\begin{align}\label{suffi_lem_G^i}
G^i(x,y) = |y| \, \alpha\!\left(|x|^2, \frac{\langle x,y\rangle}{|y|}\right) y^i + |y|^2 \, \beta\!\left(|x|^2, \frac{\langle x,y\rangle}{|y|}\right) x^i, \quad i=1,\dots,n.
\end{align}
Then the vector field
\[
G = y^i\frac{\partial}{\partial x^i} - 2G^i(x,y)\frac{\partial}{\partial y^i}
\]
is a spherically symmetric spray.
\end{lemma}

\begin{proof}
We need to verify two properties: 1. \(G\) is a spray, i.e., its coefficients \(G^i\) are positively homogeneous of degree \(2\) in \(y\); 2. \(G\) is spherically symmetric, i.e., it is invariant under the diagonal action of the orthogonal group \(O(n)\).

\textbf{Step 1: Homogeneity.}
For any \(\lambda > 0\), we compute
\begin{align*}
G^i(x,\lambda y)
&= |\lambda y| \, \alpha\!\left(|x|^2, \frac{\langle x,\lambda y\rangle}{|\lambda y|}\right) (\lambda y)^i + |\lambda y|^2 \, \beta\!\left(|x|^2, \frac{\langle x,\lambda y\rangle}{|\lambda y|}\right) x^i \\
&= \lambda^2 G^i(x, y).
\end{align*}
Thus \(G^i\) is positively homogeneous of degree \(2\) in \(y\), so \(G\) satisfies the spray condition.

\textbf{Step 2: Spherical symmetry.}
For any orthogonal transformation \(U \in O(n)\), define \(\Psi_U: \mathbb{R}^n \times (\mathbb{R}^n \backslash \{0\}) \to \mathbb{R}^n \times (\mathbb{R}^n \setminus \{0\})\) by \(\Psi_U(x,y) = (Ux, Uy)\). The map is well-defined because \(Uy \neq 0\) whenever \(y \neq 0\).

We must show that \((\Psi_U)_* G= G\) for all \(U = \left( U^j_i \right) \in O(n)\). The tangent map of \(\Psi_U\) acts on basis vectors as
\[
d\Psi_U\left( \frac{\partial}{\partial x^i} \right) = U^j_i \frac{\partial}{\partial x^j}, \qquad
d\Psi_U\left( \frac{\partial}{\partial y^i} \right) = U^j_i \frac{\partial}{\partial y^j}.
\]

First, compute \(G\) at the point \(\Psi_U^{-1}(x,y) = (U^{-1}x, U^{-1}y)\):
\[
G|_{(U^{-1}x, U^{-1}y)} = (U^{-1}y)^i \frac{\partial}{\partial x^i} - 2G^i(U^{-1}x, U^{-1}y) \frac{\partial}{\partial y^i}.
\]

Pushing forward by \((\Psi_U)_*\) gives
\begin{align*}
(\Psi_U)_* G|_{(x,y)}
&= (U^{-1}y)^i d\Psi_U\!\left( \frac{\partial}{\partial x^i} \right) - 2G^i(U^{-1}x, U^{-1}y) d\Psi_U\!\left( \frac{\partial}{\partial y^i} \right)
\\
&= (U^{-1}y)^i U^j_i \frac{\partial}{\partial x^j} - 2G^i(U^{-1}x, U^{-1}y) U^j_i \frac{\partial}{\partial y^j},
\end{align*}
where first term is simplified because \((U^{-1}y)^i U^j_i = (U U^{-1} y)^j = y^j\). Hence
\begin{align}\label{lem_suffi_PsiU*G}
(\Psi_U)_* G|_{(x,y)} = y^j \frac{\partial}{\partial x^j} - 2\left( G^i(U^{-1}x, U^{-1}y) U^j_i \right) \frac{\partial}{\partial y^j}. \tag{1}
\end{align}

It remains to show that
\begin{align}\label{lem_suffi_G^i_UxUy}
G^i(U^{-1}x, U^{-1}y) U^j_i = G^j(x,y). \tag{2}
\end{align}

Substituting the explicit form of \(G^i\) in \eqref{suffi_lem_G^i} yields
\begin{align*}
G^i(U^{-1}x, U^{-1}y)
& = |U^{-1}y| \, \alpha\!\left( |U^{-1}x|^2, \frac{\langle U^{-1}x, U^{-1}y\rangle}{|U^{-1}y|} \right) (U^{-1}y)^i \\
& \quad + |U^{-1}y|^2 \, \beta\!\left( |U^{-1}x|^2, \frac{\langle U^{-1}x, U^{-1}y\rangle}{|U^{-1}y|} \right) (U^{-1}x)^i.
\end{align*}
Since orthogonal transformations preserve lengths and inner products, we have
\[
|U^{-1}y| = |y|,\quad |U^{-1}x|^2 = |x|^2,\quad \frac{\langle U^{-1}x, U^{-1}y\rangle}{|U^{-1}y|} = \frac{\langle x,y\rangle}{|y|}.
\]
Therefore
\[
G^i(U^{-1}x, U^{-1}y) = |y| \, \alpha\!\left( |x|^2, \frac{\langle x,y\rangle}{|y|} \right) (U^{-1}y)^i + |y|^2 \, \beta\!\left( |x|^2, \frac{\langle x,y\rangle}{|y|} \right) (U^{-1}x)^i.
\]
Multiplying by \(U^j_i\) and summing over \(i\) gives
\[
G^i(U^{-1}x, U^{-1}y) U^j_i = |y| \alpha (U^{-1}y)^i U^j_i + |y|^2 \beta (U^{-1}x)^i U^j_i.
\]

Now since \((U^{-1}y)^i U^j_i = y^j\) and \((U^{-1}x)^i U^j_i = x^j\), we have
\[
G^i(U^{-1}x, U^{-1}y) U^j_i = |y| \alpha y^j + |y|^2 \beta x^j = G^j(x,y),
\]
which establishes \eqref{lem_suffi_G^i_UxUy}.
Substituting \eqref{lem_suffi_G^i_UxUy} into \eqref{lem_suffi_PsiU*G} yields
\[
(\Psi_U)_* G|_{(x,y)} = y^j \frac{\partial}{\partial x^j} - 2G^j(x,y) \frac{\partial}{\partial y^j} = G|_{(x,y)}.
\]

Hence \(G\) is spherically symmetric.
\end{proof}

\begin{proofof}{\ref{maintheorem1}}
It follows immediately from Lemma \ref{lem_spherical-spray_necessi} and \ref{lem_spherical-spray_suffi}.
\end{proofof}

Recall that a \textit{Finsler metric} on $M$ is a non-negative function $F: TM \longrightarrow [0, +\infty)$ satisfying:
\begin{enumerate}[(i)]
\item $F$ is smooth on $TM_0 := TM \setminus \{0\}$;
\item $F(x,\lambda y) = \lambda F(x,y)$ for all $\lambda > 0$ (positive homogeneity of degree one);
\item For every $y \neq 0$, the matrix $g_{ij}(x,y) := \frac{1}{2}[F^2]_{y^i y^j}$ is positive definite.
\end{enumerate}
Every Finsler metric induces a spray via
\[
G^i = \frac{1}{4} g^{ij}\bigl\{ [F^2]_{y^j x^k} y^k - [F^2]_{x^j} \bigr\},
\]
where $g^{ij}$ is the inverse of $g_{ij}$.

Let $\Pi = \operatorname{span}\{y,u\}$ be a $2$-dimensional plane in $T_xM$ with $y,u \in T_xM\setminus\{0\}$.
The \textit{flag curvature} $\mathcal{K}(\Pi, y)$ with flagpole $y$ is defined as
\[
\mathcal{K}(\Pi, y) = \frac{g_y(u, R_y(u))}{g_y(y,y)g_y(u,u) - [g_y(y,u)]^2},
\]
where $g_y = g_{ij}(x,y)dx^i\otimes dx^j$ and $R_y(u) = R^i_{\,k} dx^k \otimes \frac{\partial}{\partial x^i}$, with $R^i_{\,k}$ given by \eqref{B1}.

With these preparations, we can now derive a criterion for the metrizability of a spherically symmetric spray.
\begin{corollary}\label{maincorollary}
Let $G$ be a spherically symmetric spray on a domain $\Omega \subset \mathbb{R}^n$, ($n\geq 3$). If $G$ is of non-zero isotropic curvature and can be induced by a Finsler metric $F$, then  $F$ must be a spherically symmetric Finsler metric.
\end{corollary}
\begin{proof}
By Theorem \ref{maintheorem1} and the curvature expression \eqref{B1}, a direct computation shows $R^m_{\ m}$ is a spherically symmetric function.
Since $G$ is induced by a Finsler metric $F$ and with non-zero isotropic curvature, we have
\[ R^m_{\ m} = (n-1)k(x) F^2, \]
where $k(x) \neq 0$ is a scalar function of $x$. Since $n\geq 3$, by Schur's Lemma, it follows that $k(x)= \lambda = const.$ Then
\[ F^2 = \frac{1}{\lambda(n-1)} R^m_{\ m}, \]
which implies $F$ is a spherically symmetric function and hence a  spherically symmetric Finsler metric.
\end{proof}

\section{Projectively flat spherically symmetric sprays of isotropic curvature}\label{section4}
In this section we first derive the explicit formula for the Riemann curvature tensor of a projectively flat spherically symmetric spray, which is needed for the classification.

As stated in Introduction, a spray $G$ is of \textit{isotropic curvature} if \eqref{A5} and \eqref{A6} hold.
A direct computation shows that $R= \frac{1}{n-1}R^m_{\ m}$ in this case, i.e.,
a spray $G$ is of \textit{isotropic curvature} if
   \begin{equation}\label{B2}
	R^i_{\ k}=\frac{1}{n-1}R^m_{\ m}\delta^i_k-\frac{1}{2(n-1)}(R^m_{\ m})_{y^k}y^i.
   \end{equation}
In particular, spray $G$ is of zero curvature if $R^i_{\ k}=0$.

Let $G$ be a projectively flat spherically symmetric spray with $G^i = |y| p(r,s) y^i$, where $r=|x|^2$ and $s=\langle x,y\rangle/|y|$.
We need the following partial derivatives (computations are straightforward):
\begin{align*}
	\frac{\partial G^i}{\partial x^k}&=y^i\vert y\vert\left(\frac{\partial r}{\partial x^k}p_r+\frac{\partial s}{\partial x^k}p_s\right),\qquad
	\frac{\partial G^i}{\partial y^k}=\left(\frac{y_k}{\vert y\vert}y^i+\vert y\vert\delta^i_k\right)p+y^i\vert y\vert\frac{\partial s}{\partial y^k}p_s,\\
	\frac{\partial^2 G^i}{\partial x^j\partial y^k} %& =	\frac{\partial}{\partial x^j}\left(\left(\frac{y_k}{\vert y\vert}y^i+\vert y\vert\delta^i_k\right)p+y^i\vert y\vert\frac{\partial s}{\partial y^k}p_s\right)\\
	&=\left(\frac{y_k}{\vert y\vert}y^i+\vert y\vert\delta ^i_k\right)\left(\frac{\partial r}{\partial x^j}p_r+\frac{\partial s}{\partial x^j}p_s\right)+y^i\vert y\vert\left(\frac{\partial r}{\partial x^j}\frac{\partial s}{\partial y^k}p_{rs}+\frac{\partial s}{\partial x^j}\frac{\partial s}{\partial y^k}p_{ss}+	\frac{\partial^2 s}{\partial x^j\partial y^k}p_s\right),\\
	\frac{\partial^2 G^i}{\partial y^j\partial y^k} %&=\frac{\partial}{\partial y^j}\left(\left(\frac{y_k}{\vert y\vert}y^i+\vert y\vert\delta^i_k\right)p+y^i\vert y\vert\frac{\partial s}{\partial y^k}p_s\right)\\
	&=\left(\frac{y_k}{\vert y\vert}\delta^i_j+\frac{y_j}{\vert y\vert}\delta^i_k+y^i\frac{\delta_{kj}\vert y\vert^2-y_ky_j}{\vert y\vert^3}\right)p+y^i\vert y\vert\frac{\partial s}{\partial y^k}\frac{\partial s}{\partial y^j}p_{ss}\\
	&\quad+\left[\frac{y^i}{\vert y\vert}\left(y_j\frac{\partial s}{\partial y^k}+y_k\frac{\partial s}{\partial y^j}\right)+\vert y\vert\left(\delta^i_k\frac{\partial s}{\partial y^j}+\delta^i_j\frac{\partial s}{\partial y^k}\right)+y^i\vert y\vert\frac{\partial^2s}{\partial y^ky^j}\right]p_s,
	\end{align*}
	where
	   \begin{equation*}
		\begin{aligned}
			\frac{\partial r}{\partial x^k}&=2x_k,\qquad\frac{\partial s}{\partial x^k}=\frac{y_k}{\vert y\vert},\qquad\frac{\partial s}{\partial y^k}=\frac{x_k\vert y\vert^2-y_k\langle x,y\rangle}{\vert y\vert^3},\\
			\frac{\partial^2 s}{\partial x^jy^k}&=\frac{\delta_{kj}\vert y\vert^2-y_ky_j}{\vert y\vert^3},\quad\frac{\partial^2 s}{\partial y^jy^k}=-\frac{\langle x,y\rangle}{\vert y\vert^3}\delta_{kj}+\frac{3\langle x,y\rangle y_ky_j-\vert y\vert^2(x_ky_j+x_jy_k)}{\vert y\vert^5}.
		\end{aligned}
	\end{equation*}

After substituting into the formula for the Riemann curvature tensor \eqref{B1},
a lengthy but direct calculation yields the following compact expression:
\begin{align}\label{B3}
\begin{split}
R^i_{\ k}&=|y|^2(p^2-2sp_r-p_s)\delta^i_k\\
&\quad -\Big\{|y|(pp_s+2sp_{rs}-4p_r+p_{ss})x_k
-\big[(sp+1)p_s+2s(sp_{rs}-p_r)+sp_{ss}-p^2\big]y_k\Big\}y^i.
\end{split}
\end{align}

Now we prove Theorem \ref{maintheorem2}. The proof is decomposed into two propositions.
\begin{prop}\label{prop4.1}
Let $G$ be a projectively flat spherically symmetric spray with $G^i = |y|\,p(r,s)\,y^i$, $r=|x|^2$, $s=\langle x,y\rangle/|y|$.
Then $G$ is of isotropic curvature if and only if
there exist smooth functions $u(r-s^2)$ and $v(r)$ such that
\begin{equation}\label{C1}
p(r, s)=\left[\int \frac{u(r-s^2)}{s^2}ds+v(r)\right]s.
\end{equation}
\end{prop}
\begin{proof}
From \eqref{B3} we obtain
   \begin{align*}
	R^m_{\ m}&= \vert y\vert^2(n-1)(p^2-2sp_r-p_s),\\
	(R^m_{\ m})_{y^k}& =\vert y\vert(n-1)(2pp_s-2sp_{rs}-2p_r-p_{ss})x_k\\
	& \quad+(n-1)\left\{2\left[p^2-(sp+1)p_s+s(sp_{rs}-p_r)\right]+sp_{ss}\right\}y_k.
   \end{align*}
%Then
%   \begin{align}\label{C3}
%   &\quad \frac{1}{n-1}R^m_{\ m}\delta^i_k-\frac{1}{2(n-1)}(R^m_{\ m})_{y^k}y^i \notag \\
%   &=\vert y\vert^2(p^2-2sp_r-p_s)\delta^i_k-\Big\{\vert y\vert(pp_s-sp_{rs}-p_r-\frac{1}{2}p_{ss})x_k \notag\\
%   &\quad+[p^2-(sp+1)p_s+s(sp_{rs}-p_r)+\frac{1}{2}sp_{ss}]y_k\Big\}y^i.
%   \end{align}
Using these, a direct computation shows that
\begin{align}
R^i_{\ k}-\left[\frac{1}{n-1}R^m_{\ m}\delta^i_k-\frac{1}{2(n-1)}(R^m_{\ m})_{y^k}y^i\right]
=3(sp_{rs}-p_r+ \frac{1}{2} p_{ss} )(sy_k-\vert y\vert x_k)y^i.\label{C4}
\end{align}
By definition, a spray is of isotropic curvature if and only if \eqref{B2} holds, which is equivalent to the vanishing of the left-hand side of \eqref{C4}. Since $sy_k-|y|x_k$ is not identically zero for non-collinear $x,y$, we obtain the necessary and sufficient condition
\begin{equation}\label{C5}
		s p_{rs}-p_r+ \frac{1}{2} p_{ss}=0.
	\end{equation}
%Now	we prove the necessity first.
%If spray $G$ is of isotropic curvature, by \eqref{B2} and \eqref{C4}, we have
%	\begin{equation*}
%		3(sp_{rs}-p_r+ \frac{p_{ss}}{2})(sy_k-\vert y\vert x_k)y^i=0.
%	\end{equation*}
%If vector $y$ is parallel to vector $x$, the above equation is satisfied directly. Then we consider the general case when vector $y$ and vector $x$ is not parallel, we get
%	\begin{equation}\label{C5}
%		sp_{rs}-p_r+ \frac{p_{ss}}{2}=0.
%	\end{equation}

To solve \eqref{C5}, set $p(r,s)=s f(r,s)$. Then
\[ p_r=sf_r, \quad p_s=f+sf_s, \quad p_{rs}=f_r+sf_{rs}, \quad p_{ss}=2f_s+sf_{ss}, \]
and \eqref{C5} becomes
	\begin{equation}\label{C6}
		s^2 f_{rs} + f_s + \frac{sf_{ss}}{2} = 0.
	\end{equation}
Let $\hat{f}= f_s$. Then \eqref{C6} reduces to the first-order linear PDE
	\begin{equation}\label{C7}
		s^2\hat{f}_r+\hat{f}+\frac{s\hat{f}_s}{2}=0.
	\end{equation}

Its characteristic system is
	\begin{equation*}
		\frac{dr}{s^2}=\frac{ds}{s/2}=\frac{d\hat{f}}{-\hat{f}},
	\end{equation*}
which we solve on the region $s\neq0$. From $dr/s^2 = ds/(s/2)$ we get $dr=2s\,ds$, so $r=s^2+c_1$, i.e. $c_1=r-s^2$.
From $ds/(s/2)=d\hat{f}/(-\hat{f})$ we get $2\ln|s| = -\ln|\hat{f}|+\ln c_2$, hence $\hat{f}=c_2/s^2$.
Along characteristics, $c_2$ is a function of $c_1$, thus
\[
\hat{f}(r,s) = \frac{u(r-s^2)}{s^2}
\]
for some smooth function $u$. Then $f_s = u(r-s^2)/s^2$ implies, for $s\neq0$,
\[
 f(r,s) = \int \frac{u(r-s^2)}{s^2}\,ds + v(r),
\]
with $v(r)$ arbitrary.
Since \(p\) is the projective factor of a smooth spray, \(p\) is smooth on \(TM\setminus\{0\}\). The expression \(p=sf\) obtained above is smooth for \(s\neq0\) and extends smoothly to \(s=0\) by the smoothness of the spray.
The converse follows by direct substitution.
\end{proof}

We now consider the zero curvature case. Recall that a spray $G$ has zero curvature if $R^i_{\ k}=0$.
Using \eqref{B3} and the linear independence of $x^i$ and $y^i$ for non-collinear vectors, the condition $R^i_{\ k}=0$ splits into three equations:
   \begin{align}
	p^2-2sp_r-p_s &=0,\label{C8}\\
	pp_s+2sp_{rs}-4p_r+p_{ss} &=0,\label{C9}\\
	(sp+1)p_s+2s(sp_{rs}-p_r)+sp_{ss}-p^2 &=0.\label{C10}
   \end{align}
A direct calculation shows that $s\times\eqref{C9}-\eqref{C8}$ reduces to \eqref{C10}; hence the system \eqref{C8}--\eqref{C10} is equivalent to  \eqref{C8} and \eqref{C9} alone:
\begin{equation}\label{0_Rik}
R^i_{\ k}=0 \quad\Longleftrightarrow\quad \text{\eqref{C8} and \eqref{C9} hold}.
\end{equation}
%A direct substitution of the general isotropic  curvature solution \eqref{C1} into \eqref{C8} and \eqref{C9} leads to complicated conditions on $u$ and $v$; it is more efficient to solve these equations directly for $p(r,s)$.

\begin{prop}\label{prop4.2}
Let $G$ be a projectively flat spherically symmetric spray  with $G^i=|y|p(r,s)y^i$. Then $G$ has zero curvature if and only if $p(r,s)$ takes one of the following forms:
\begin{enumerate}[\rm (i)]
\item \label{prop4.2_casei} \(p\equiv0\) (the zero spray);
\item \begin{equation}\label{C11}
p(r,s)= \frac{s\pm\sqrt{s^2-r+c}}{c-r},
\end{equation}
where $c>0$ is an arbitrary constant, and the spray is defined on $B_{\sqrt{c}}^n(0)\times(\mathbb{R}^n\backslash\{0\})$, with $B_{\sqrt{c}}^n(0)$ the open ball of radius $\sqrt{c}$.
\end{enumerate}
\end{prop}
\begin{proof}
We need to solve \eqref{C8} and \eqref{C9}. The case $p\equiv0$ trivially satisfies both. Assume $p\not\equiv0$.
Differentiate \eqref{C8} with respect to $s$ yields
 \[ 2p p_s-2p_r-2 s p_{rs}-p_{ss}=0.\]
Substituting into \eqref{C9} eliminates $p_{rs}$ and yields $pp_s-2p_r=0$. Then using \eqref{C8} to eliminate $p_r$ gives
\[ p^2- s p p_s - p_s=0. \]
Since $p\not\equiv0$, the set where $p\neq0$ is a non-empty open subset. On this set we may divide by $p$ and obtain
\[
\left(\frac{s}{p}+\frac{1}{2p^2}\right)_s = \frac{p^2-spp_s-p_s}{p^3}=0,
\]
so there exists a function $h(r)$ such that
\[
\frac{s}{p}+\frac{1}{2p^2}=h(r) \quad\Longrightarrow\quad 2h(r)p^2-2sp-1=0.
\]
Since the zero set of \(p\) has empty interior (otherwise \(2h(r)p^2-2sp-1=0\) would fail on an open set), the relation extends to the zero set of \(p\) by continuity.

If $h(r_0)=0$ for some $r_0$, then $2h(r_0)p^2-2sp-1=0$ reduces to $-2sp-1=0$ for all $s$, which is impossible at $s=0$. Hence $h(r)\neq0$ everywhere. Evaluating at \(s=0\) gives \(2h(r)p(r,0)^2=1\); in particular \(p(r,0)\neq0\), so \(h(r)=1/(2p(r,0)^2)>0\). Solving the quadratic gives
\begin{equation}
p(r,s)=\frac{s\pm\sqrt{s^2+2h(r)}}{2h(r)}. \label{C13}
\end{equation}
Substituting this into \eqref{C8} leads to
\[
\frac{(h'(r)+\frac12)\bigl[\pm(s^2+h(r))+s\sqrt{s^2+2h(r)}\bigr]s}{h(r)^2\sqrt{s^2+2h(r)}}=0.
\]
For $s\neq0$, since $\sqrt{s^2+2h(r)}>|s|$ and $h(r)>0$, the bracket is non-zero, so we must have $h'(r)+\frac12=0$. By continuity, this holds for all $s$, i.e. $h(r)=-\frac{r}{2}+\frac{c}{2}$ with $c$ constant. Positivity $h(r)>0$ implies $c>0$ and $|x|^2< c$. Substituting back into \eqref{C13} gives \eqref{C11}. The sufficiency is verified by direct substitution.
\end{proof}

\begin{remark}\label{remark} Here, \eqref{C11} is equivalent to \eqref{C1} with $u(r-s^2)=\mp\frac{1}{\sqrt{c-r+s^2}}$ and $v(r)=\frac{s\pm\sqrt{c-r+s^2}}{s(c-r)}\pm \int\frac{1}{s^2\sqrt{c-r+s^2}}ds$.
\end{remark}

\begin{proofof}{\ref{maintheorem2}}
It follows immediately from Proposition \ref{prop4.1} and Proposition \ref{prop4.2}.
\end{proofof}

\begin{corollary} \label{example1}
Every projectively flat spherically symmetric spray of zero curvature  can be induced by a Finsler metric $F$.
\end{corollary}
\begin{proof}

By Theorem \ref{maintheorem2}, a zero-curvature spray either has $G^i=0$ (induced by a Minkowski metric) or $G^i = |y| \frac{s\pm\sqrt{c-r+s^2}}{c-r} y^i$. In the latter case, the corresponding metric is Berwald's metric (or its reverse) up to scaling:
\[
F(x,y)=\frac{\bigl(\sqrt{(c-|x|^2)|y|^2+\langle x,y\rangle^2}\pm\langle x,y\rangle\bigr)^2}{(c-|x|^2)^2\sqrt{(c-|x|^2)|y|^2+\langle x,y\rangle^2}},
\]
with its projective factor
	\begin{equation*}
			P= \frac{\langle x,y\rangle}{c-\vert x\vert^2} \pm \frac{\sqrt{ (c-\vert x\vert^2)\vert y\vert^2  + \langle x,y\rangle^2}}{c-\vert x\vert^2}.
	\end{equation*}
\end{proof}

Unlike the zero curvature case, there exist many spherically symmetric sprays which cannot be induced by any Finsler metric.
\begin{ex}\label{example2}
Here are some examples of projectively spherically symmetric sprays with non-zero isotropic curvature.
\begin{enumerate}[\rm (i)]
\item Let $u(r-s^2)=s^2-r+C_1$, $v(r)=r+C_2$, where $C_1,C_2,C_3$ are constants. Then
\[
p(r,s)=\bigl[\int\frac{-r+s^2+C_1}{s^2}ds+r+C_2\bigr]s = s^2+(r+\widetilde{C_2})s+r-C_1,
\]
with $\widetilde{C_2}=C_2+C_3$.
By Berwald's theorem on projectively flat Finsler metrics of constant curvature, if the corresponding  projective flat spray $G$ ($G^i = P y^i = |y|p(r,s) y^i$) is induced by a Finsler metric $F$,
then $F$ must be of constant flag curvature $\lambda$ because $G$ is of isotropic curvature.
Then by Berwald's result (\cite{Be2})
\[ \lambda F^2 = P^2 - P_{x^k} y^k, \]
$F$ must be a spherically symmetric Finsler metric.
However,
this quadratic function of $s$ does not belong to the classification of projectively flat Finsler metrics of constant flag curvature (see \cite{4, 11}).
\item Let $u(r-s^2)=-\frac{1}{2\sqrt{s^2-r+1}}$, $v(r)=\frac{1}{2(1-r)}$. Then
\[
p(r,s)=\frac{\sqrt{s^2-r+1}+s}{2(1-r)}+Cs,
\]
which for $C=0$ gives the projective factor of the Funk metric
\[
F(x,y)=\frac{\sqrt{\langle x,y\rangle^2+|y|^2-|x|^2|y|^2}+\langle x,y\rangle}{1-|x|^2},
\]
of flag curvature $\mathcal{K} =-\frac14$.
\end{enumerate}
\end{ex}

\section{Projectively flat spherically symmetric sprays of weakly isotropic curvature}\label{section5}
In this section, we prove Theorem \ref{maintheorem3} and construct  new sprays of weakly isotropic curvature which are not induced by any Finsler metric.

We first recall the definition. A spray $G$ is called of \textit{weakly isotropic curvature} if it is of scalar curvature and there exist a homogeneous function $\Gamma=\Gamma(x,y)$ (positively homogeneous of degree $1$ in $y$) and a $1$-form $\theta=\theta_i(x)y^i$ such that
\begin{align}
 \Gamma_{y^k} R&=\tau_k\Gamma ,\label{A9}\\
\tau_k&=\frac12 R_{y^k}+\frac32\Bigl( \Gamma_{y^k}\theta-\Gamma\theta_k\Bigr)\label{A10}.
\end{align}

For a spherically symmetric spray, $\theta$ can be expressed as $\theta=a(r)\langle x,y\rangle$, where $a(r)$ is a smooth function. Moreover, $\Gamma$ must be $1$-homogeneous, so we write $\Gamma=|y|\gamma(r,s)$. The following lemma to determine whether a spray of weakly isotropic curvature is induced by a Finsler metric.
\begin{lem}\cite{8}\label{whether induce}
	Let $G$ be a spray of weakly isotropic curvature and induced by a Finsler metric $F$ on $M$, then
	\begin{equation*}
		\mathcal{K}=0\quad or\quad \Gamma=\lambda(x)F,
	\end{equation*}
	where $\lambda(x)$ is a scalar function.
\end{lem}
The proof of this lemma can be found in the proof of Theorem 1.2 in \cite{8}. Now we derive the equivalent conditions for a projectively flat spherically symmetric spray to be of weakly isotropic curvature.
\begin{theorem}\label{maintheorem3}
Let $G$ be a projectively flat spherically symmetric spray with $G^i=|y|p(r,s)y^i$, $r=|x|^2$, $s=\langle x,y\rangle/|y|$. Then $G$ is of weakly isotropic curvature if and only if there exist a smooth function $a(r)$ and a smooth function $\gamma(r,s)$ such that
\begin{equation}\label{A11}
\begin{cases}
2s\gamma p_{rs}+\gamma p_{ss}+(2sp_r-p^2+p_s)\gamma_s-(-pp_s+4p_r)\gamma=0,\\[2mm]
a(r)s\gamma_s-2sp_{rs}-a(r)\gamma+2p_r-p_{ss}=0.
\end{cases}
\end{equation}
\end{theorem}
\begin{proof}
Since a projectively flat spray is of scalar curvature, using \eqref{B3}, we have
\begin{align*}
R&=|y|^2(p^2-2sp_r-p_s),\\
\tau_k&=|y|(pp_s+2sp_{rs}-4p_r+p_{ss})x_k-\big[(sp+1)p_s+2s(sp_{rs}-p_r)+sp_{ss}-p^2\big]y_k.
\end{align*}
Set $\Gamma=|y|\gamma(r,s)$, $\theta=a(r)\langle x,y\rangle$. Then
\[
\Gamma_{y^k}=\frac{y_k}{|y|}\gamma+|y|\gamma_s s_{y^k},\qquad
\theta_k=\frac{\partial\theta}{\partial y^k}=a(r)x_k,
\]
and
\[
R_{y^k}= 2 y_k(p^2-2sp_r-p_s)+|y|^2(2pp_s s_{y^k}-2s_{y^k}p_r-2sp_{rs}s_{y^k}-p_{ss}s_{y^k}).
\]
Substituting these into \eqref{A9} and \eqref{A10} and simplifying using the expressions for $s_{y^k}$, we obtain after a lengthy but straightforward computation:
\begin{align}
&|y|(sy_k-|y|x_k)\big[2s\gamma p_{rs}+\gamma p_{ss}+(2sp_r-p^2+p_s)\gamma_s-(-pp_s+4p_r)\gamma\big]=0,\\
&\frac32(sy_k-|y|x_k)\big[a(r)s\gamma_s-2sp_{rs}-a(r)\gamma+2p_r-p_{ss}\big]=0.
\end{align}
Since $sy_k-|y|x_k$ does not vanish identically on the open dense set of non-collinear $(x,y)$, the brackets vanish there. By continuity, they vanish everywhere, so these equations are equivalent to the system \eqref{A11}.
\end{proof}

It is difficult to solve \eqref{A11} in full generality, but we can find some special solutions.
Weakly isotropic curvature is closely related to weakly isotropic flag curvature in Finsler geometry \cite{8}.
It is also natural to ask whether a spray of weakly isotropic curvature can be induced by a Finsler metric.
In \cite{5}, B. Li studied Finsler metrics of weakly isotropic flag curvature and proved that if the dimension of the manifold satisfies \(\dim(M)\ge 3\), then such a metric must be a Randers metric; otherwise, it must be a metric of constant flag curvature.
The following example was first given in \cite{8} (for \(\varepsilon=1\)).
\begin{ex}\label{ex1}
Let the projectively flat spherically symmetric spray $G$ be defined by
\begin{equation}\label{D6}
G^i=Py^i,\quad P=\frac{\mu|y|}{\sqrt{\varepsilon-|x|^2}}+\frac{2\langle x,y\rangle}{\varepsilon-|x|^2},
\end{equation}
where $\varepsilon,\mu$ are constants, $\mu>0$. Now we determine whether it is induced by a Finsler metric.
By \eqref{A4}, \eqref{A5}, \eqref{B3} and a direct computation, we get
\begin{equation}\label{D7}
	R=\frac{(\mu^2-2)\vert y\vert^2}{\varepsilon-\vert x\vert^2}+\frac{3\mu\vert y\vert\langle x,y\rangle}{(\varepsilon-\vert x\vert^2)^{\frac{3}{2}}},\quad \tau_k=\frac{(\mu^2-2)y_k}{\varepsilon-\vert x\vert^2}+\frac{3\mu\langle x,y\rangle y_k}{\vert y\vert(\varepsilon-\vert x\vert^2)^{\frac{3}{2}}},
\end{equation}
where $y_k=y^k$ and $x_k=x^k$. Differentiating $R$ with respect to $y^k$ in \eqref{D7} yields
\begin{equation*}
	R_{y^k}=\frac{2(\mu^2-2)y_k}{\varepsilon-\vert x\vert^2}+\frac{3\mu\vert y\vert x_k}{(\varepsilon-\vert x\vert^2)^{\frac{3}{2}}}+\frac{3\mu\langle x,y\rangle y_k}{\vert y\vert(\varepsilon-\vert x\vert^2)^{\frac{3}{2}}}.
\end{equation*}
Then
\begin{equation*}
	\tau_k=\frac{1}{2}R_{y^k}+\frac{3\mu}{2\vert y\vert(\varepsilon-\vert x\vert^2)^{\frac{3}{2}}}\left(\langle x,y\rangle y_k-\vert y\vert^2 x_k\right).
\end{equation*}
In this case,
\begin{equation*}
	\Gamma\theta=\frac{\mu\vert y\vert\langle x,y\rangle}{(\varepsilon-\vert x\vert^2)^{\frac{3}{2}}},
\end{equation*}
by choosing
\begin{equation}\label{D8}
	\Gamma=\vert y \vert\quad \text{and} \quad \theta=\frac{\mu\langle x,y\rangle}{(\varepsilon-\vert x\vert^2)^{\frac{3}{2}}},
\end{equation}
then \eqref{A9} and \eqref{A10} are satisfied obviously, and it is clear that the spray in \eqref{D6} is of weakly isotropic curvature. By Lemma \ref{whether induce}, we know if spray $G$ in \eqref{D6} is induced by a Finsler metric, then there exist a scalar function $\lambda(x)$ such that
\begin{equation*}
	\Gamma=\lambda(x)F.
\end{equation*}
By \eqref{D8}, $F$ is conformal to Euclidean metric and then it is a Riemannian metric. It is well-known that spray of Riemannian metric is quadratic in $y$, so this is impossible in \eqref{D6}, hence spray $G$ in \eqref{D6} is not induced by any Finsler metric.
\end{ex}

The following example is new.
\begin{ex}
Let $G$ be defined by
\begin{equation}\label{ex2}
G^i=Py^i,\quad P=\frac{2b\sqrt[4]{1+c|x|^2}}{c}\,\alpha_c,
\end{equation}
where
\[
\alpha_c=\frac{\sqrt{|y|^2+c(|x|^2|y|^2-\langle x,y\rangle^2)}}{1+c|x|^2},
\]
and $b,c\neq0$ are constants. Now we determine whether it is induced by a Finsler metric or not. By
\eqref{A4}, \eqref{A5}, \eqref{B3} and a direct computation, we get
\begin{align}
	R&=\frac{4b^2\sqrt{1+c\vert x\vert^2}}{c^2}\alpha_c^2+\frac{3b\langle x,y\rangle}{(1+c\vert x\vert^2)^{3/4}}\alpha_c,\label{D11}\\
	\tau_k&=\frac{4b^2\sqrt{1+c\vert x\vert^2}}{c^2}\alpha_c(\alpha_c)_{y^k}+\frac{3b\langle x,y\rangle}{(1+c\vert x\vert^2)^{3/4}}(\alpha_c)_{y^k}\nonumber
\end{align}
Differentiating $R$ with respect to $y^k$ in \eqref{D11} yields
\begin{equation*}
	R_{y^k}=\frac{8b^2\sqrt{1+c\vert x\vert^2}}{c^2}\alpha_c(\alpha_c)_{y^k}+\frac{3b}{(1+c\vert x\vert^2)^{3/4}}\{x_k\alpha_c+\langle x,y\rangle(\alpha_c)_{y^k}\}
\end{equation*}
Then
\begin{equation*}
	\tau_k=\frac{1}{2}R_{y^k}+\frac{3b}{2(1+c\vert x\vert^2)^{3/4}}\{\langle x,y\rangle(\alpha_c)_{y^k}-x_k\alpha_c\}.
\end{equation*}
In this case,
\begin{equation*} 
	\Gamma=\alpha_c\quad \text{and} \quad \theta=\frac{b\langle x,y\rangle}{(1+c\vert x\vert^2)^{3/4}},
\end{equation*}
then \eqref{A9} and \eqref{A10} are satisfied, and it is clear that the spray in \eqref{ex2} is of weakly isotropic curvature. By the same proof in Example \ref{ex1}, spray $G$ in \eqref{ex2} is not induced by any Finsler metric.
\end{ex}

\section*{Declarations}

\noindent\textbf{Competing interests:} The authors declare no conflict of interest.

\noindent\textbf{Availability of data and material:} Not applicable.

\noindent\textbf{Declaration of Generative AI and AI-assisted technologies in the writing process:} AI tools were used exclusively to assist with language polish and expression. The authors conducted all subsequent editing, verified the validity of the material, and assume sole responsibility for the final manuscript.

\end{document}